\documentclass[11pt]{amsart}

\usepackage{amssymb,amsmath,amscd,graphicx,
latexsym,amsthm}
\usepackage{amssymb,latexsym,amsmath,amscd,graphicx}
\usepackage[mathscr]{eucal}
  \usepackage[all]{xy}
\usepackage[a4paper,top=3.5cm,bottom=3.5cm,left=3cm,right=3cm]{geometry}

\def\Z{{\mathbb Z}}
\def\N{{\mathbb N}} 
\def\Q{{\mathbb Q}}
\def\C{{\mathbb C}}                  
\def\P{{\mathbb P}}
\def\K{{\mathbb K}}

\def\OO{{\mathcal O}}

\def\FF{\mathcal{F}}

\def\II{\mathcal{I}}
\def\JJ{\mathcal{J}}
\def\PP{\mathcal{P}}

\def\JJ{\mathcal J}
\def\Ad{\widehat{A}}

\def\l{{\underline{\textit{l}}}}
\def\n{{\underline{\textit{n}}}}

\def\HHom{\mathcal{H}om}

\def\eps{{\epsilon}}

\def\dual{{\vee}}
\def\tensor{{\otimes}}

\def\Codim{\mathrm{codim}}
\def\Sym{\mathrm{Sym}}
\def\NS{\mathrm{NS}}
\def\Pic0{\mathrm{Pic}^0}

\def\Inf{\mathrm{Inf}}
\def\Sup{\mathrm{Sup}}
\def\max{\mathrm{max}}
\def\min{\mathrm{min}}
\def\deg{\mathrm{deg}}
\def\Char{\mathrm{char}}
\def\Supp{\mathrm{Supp}}
\def\NS{\mathrm{NS}}
\def\Tor{\mathrm{Tor}}
\def\reg{\mathrm{reg}}

\def\Ker{\mathrm{Ker}}

\def\la{{\langle}}
\def\ra{{\rangle}}

\theoremstyle{plain}

\newtheorem{theorem}{Theorem}[section]

\newtheorem{proposition/example}[theorem]{Proposition/Example}
\newtheorem{definition/theorem}[theorem]{Definition/Theorem}
\newtheorem{proposition}[theorem]{Proposition}

\newtheorem{corollary}[theorem]{Corollary}
\newtheorem{lemma}[theorem]{Lemma}

\theoremstyle{definition}
\newtheorem{definition}[theorem]{Definition}

\newtheorem{remark}[theorem]{Remark}

\newtheorem{conjecture/question}[theorem]{Conjecture/Question}

\newtheorem{remark/definition}[theorem]{Remark/Definition}
\newtheorem{notation/assumptions}[theorem]{Assumptions/Notation}

\numberwithin{equation}{section}

 \theoremstyle{remark}

\begin{document}  

\title{the basepoint-freeness threshold and syzygies of abelian varieties}

\author{Federico Caucci}
  \address{Sapienza Universit\`{a} di Roma, P.le Aldo Moro 5, I-00185 Roma, Italy}
 \email{{\tt caucci@mat.uniroma1.it}}
 \thanks{}

\maketitle

\setlength{\parskip}{.1 in}

\begin{abstract} 
We show how a natural constant introduced by Jiang and Pareschi for a polarized abelian variety encodes information about the syzygies of the section ring of the polarization. As a particular case this gives a quick and characteristic-free proof of Lazarsfeld's conjecture on syzygies of abelian varieties, originally proved by Pareschi in characteristic zero. 
 \end{abstract}

\section{Introduction}

Throughout this paper we will work with abelian varieties over an algebraically closed field $\K$. 
In \cite{jipa}, Jiang and Pareschi introduced and studied the (generic) cohomological ranks $h^i(A, \FF\la x\l \ra)$ of a (bounded complex of) $\Q$-twisted coherent sheaf on a polarized abelian variety $(A, \l)$. This defines \emph{cohomological rank functions} of $\FF$ with respect to the polarization $\l$
\[
h^i_{\FF, \l} : \Q \rightarrow \Q_{\geq 0},\footnote{In \emph{op.cit.} such functions are extended to (continuous) real functions, but in this paper we don't need this.}
\]   
as follows 
\[
h^i_{\FF, \l}(x):=h^i(A, \FF\la x\l \ra).
\]
In \emph{op.cit.} it is observed that these functions are already very interesting in the case $\FF = \II_p$, where $\II_p$ is the ideal sheaf of a closed point $p \in A$. Indeed the \emph{basepoint-freeness threshold}
\[
\eps_1(\l) := \Inf \{x \in \Q \ |\ h^1_{\II_p, \l}(x) = 0 \},\footnote{In \emph{op.cit.}\ this is denoted by $\beta(\l)$.}
\] 
has the following properties: 

\noindent (a) \emph{ $\eps_1(\l) \leq 1$ and   $\eps_1(\l) < 1$  if and only if the polarization $\l$ is basepoint-free}, i.e.\ any line bundle $L$ representing $\l$ has no base points.  \\
 (b)  
\emph{ $\eps_1(\l) < \frac{1}{2}$  
if and only if $\l$ is projectively normal, } meaning that $L$ is projectively normal for all line bundles $L$ representing the class $\l$ (\cite{jipa} Corollary E).

\noindent
 In this paper we go further on item (b), proving that $\eps_1(\l)$ indeed encodes information about the syzygies of the section algebra of $L$. In recent years syzygies of abelian varieties  has received considerable attention. On the one hand Pareschi (\cite{pa1}, see also \cite{papoII}), building partially on previous works of Kempf (\cite{ke}, \cite{ke2}), proved, in characteristic zero, Lazarsfeld's conjecture on syzygies of abelian varieties endowed with a polarization which is a multiple of a given one. This was is in turn a generalization of classical results of Koizumi and Mumford. On the other hand, more recently K\"uronya, Ito and Lozovanu (\cite{kulo}, \cite{ito}, \cite{lo}), building on previous work of Hwang-To (\cite{ht}) and Lazarsfeld-Pareschi-Popa (\cite{lapapo}), used completely different methods -- involving local positivity and Nadel vanishing theorem -- to prove (over $\mathbb C$) effective statements for the syzygies of  abelian varieties of dimension 2 and 3 endowed with \emph{any} polarization, in particular with a primitive polarization.  
 
 In this paper we show a general result, Theorem \ref{main1} below, partially generalizing (b) to higher syzygies. This provides at the same time a surprisingly quick proof of Lazarsfeld's conjecture, extending it to abelian varieties defined over a ground field of arbitrary characteristic, and a proof of the criterion of \cite{lapapo} relating local positivity and syzygies.

Turning to details, we first recall some terminology about syzygies of projective varieties.  Let $X$ be a projective variety and let $L$ be an ample line bundle on $X$. For an integer $p \geq 0$, the line bundle $L$ is said to \emph{satisfy the property} $(N_p)$ if the first $p$ steps of the minimal graded free resolution of the section algebra $R_L = \bigoplus_m H^0(X, L^m)$ over the polynomial ring $S_L = \Sym\ H^0(X, L)$ are linear (we refer to \S4 for the precise definition). Thus $(N_0)$ means that $R_L$ is generated in degree $0$ as an $S_L$-module, i.e.\ that $L$ is projectively normal (\emph{normally generated} in Mumford's terminology \cite{mum}); $(N_1)$ means that in addition the homogeneous ideal $I_{X/\P}$ of $X$ in $\P = \P(H^0(X, L)^{\dual})$ is generated by quadrics (\emph{normally presented} in \cite{mum}); $(N_2)$ means that the relations among these quadrics are generated by linear ones (this is the first non-classical condition) and so on. These notions were introduced by Green (\cite{green}) and the present terminology was introduced in \cite{grla}. 
Our main result is the following 
\begin{theorem}\label{main1}
Let $(A, \l)$ be a polarized abelian variety defined over an algebraically closed field $\K$ and let $p$ be a non-negative integer. If
\[
\eps_1(\l) < \frac{1}{p+2},
\]
then the property $(N_p)$ holds for $\l$, i.e.\ it holds for any line bundle $L$ representing $\l$. 
\end{theorem}
\begin{corollary}\label{pathm}
Let $m \in \N$. If 
\[
\eps_1(\l) < \frac{m}{p+2},
\] 
then the polarization $m\l$ satisfies the property $(N_p)$.
\end{corollary}
\begin{proof}
By definition (see \S\ 2) we have $h^1_{\II_p, m\l}(x) = h^1_{\II_p, \l}(mx)$, therefore $\eps_1(m\l) = \frac{\eps_1(\l)}{m}$. Now Theorem \ref{main1} applies to $m\l$, because $\eps_1(m\l) < \frac{1}{p+2}$. 
\end{proof}
\noindent A classical result of Koizumi (\cite{ko}) states that if $L$ is an ample line bundle on a complex abelian variety and $m \geq 3$, then $L^m$ is projectively normal (see \cite{se1}, \cite{sa} and \cite{se2} for a proof of the analogue result in positive characteristic, based on Mumford's ideas). Moreover, a well-known theorem of Mumford and Kempf says that, when $m \geq 4$, the homogeneous ideal of $A$ in the embedding given by $L^m$ is generated by quadrics (\cite{mum}, \cite{ke2} Thm 6.13), i.e.\ $L^m$ is normally presented. Based on these classical facts and motivated by a result of Green on higher syzygies for curves (\cite{green}), Lazarsfeld conjectured that, for an ample line bundle $L$ on an abelian variety, $L^m$ satisfies the property $(N_p)$ if $m \geq p+3$ (\cite{lasampl} Conjecture 1.5.1). This was proved by Pareschi (\cite{pa1}) in characteristic zero. Pareschi and Popa also proved a stronger version of it in \cite{papoII}. 

We have that Corollary \ref{pathm} gives a very quick -- and characteristic-free -- proof of Lazarsfeld's conjecture. Indeed, by (a) above,
\[
\eps_1(\l) \leq 1 < \frac{p+3}{p+2}. 
\]
Moreover, it also implies that the polarization $m\l$ satisfies the property $(N_p)$, as soon as $m \geq p + 2$ and $\l$ is basepoint-free (see \cite{papoII} for a more precise result). Indeed, if $\l$ is basepoint-free, then
\[
\eps_1(\l) < 1 = \frac{p+2}{p+2},
\]
thanks again to (a) above.

More in general, defining
\[
t(\l) := \max \{t \in \N \ |\ \eps_1(\l) \leq \frac{1}{t} \}, 
\]   
we have 
\begin{theorem}\label{syzmult}
Let $p$ and $t$ be non-negative integers with $p + 1 \geq t$. Let $\l$ be a basepoint-free polarization on $A$ such that $t(\l) \geq t$. Then the property $(N_p)$ holds for $m \l$, as soon as $m \geq p +3-t$. 
\end{theorem}

However, one of the main feature of Theorem $\ref{main1}$ is the chance to be applied to \emph{primitive} polarizations, i.e.\ those that cannot be written as a multiple of another one. This is one of the reasons why it would be quite interesting the compute, or at least bound from above, the invariant $\eps_1(\l)$ of polarized abelian varieties $(A,\l)$. In this perspective, as already mentioned, an interesting issue arises in connection with a criterion of Lazarsfeld-Pareschi-Popa (\cite{lapapo}), where they prove that: \\
\emph{ if there exists an effective $\Q$-divisor $F$ such that its multiplier ideal}  $\JJ(A, F)$ \emph{is the ideal sheaf of the identity point of the abelian variety $A$ and} $\frac{1}{p+2}\l-F$ \emph{is ample, then $\l$ satisfies the property $N_p$} (see \cite{kulo},\cite{ito},\cite{lo}). \\
Therefore one is lead to  consider the threshold
\[
r(\l) := \Inf \{r \in \Q \ |\ \exists\ \textrm{an effective $\Q$-divisor $F$ on $A$ s.t.}\ r\l - F\ \textrm{is ample and 
 }\ \JJ(A, F) = \II_0\}.\footnote{Note that this set is non empty, i.e.\ $r(\l) < + \infty$. Proof: let $k$ be a sufficiently large positive integer such that the Seshadri constant of $M = L^k$ is strictly bigger than $2\dim A$. Such a $k$ exists because of the homogeneity of the Seshadri constant. Then, by Lemma 1.2 of \cite{lapapo}, there exists an effective $\Q$-divisor $F$ on $A$ such that $\JJ(A, F) = \II_0$ and $F \equiv_{\textrm{num}} \frac{1-c}{2}M$, for some $0 < c \ll 1$. If we now take $r > \frac{1-c}{2}k$, we have that $r\l - F$ is ample.}
\] 
The relation  with the basepoint-freeness threshold is in the following Proposition, based on Nadel's vanishing.
\begin{proposition}\label{nadel} Assume $\K=\C$. Then $\eps_1(\l)\le r(\l)$.
\end{proposition}

This, combined with Theorem \ref{main1}, provides a different and simpler  proof of the criterion of \cite{lapapo}.

 Finally, we note that in the papers \cite{kulo}, \cite{ito} for dimension $2$ and \cite{lo} for dimension $3$,  the authors, in the spirit of Green's and Green-Lazarsfeld's conjectures on curves,  show explicit geometric conditions 
ensuring the property  $(N_p)$ by means  of upper bounds on  the threshold $r(\l)$ (or related invariants) and applying the criterion of \cite{lapapo}.  This suggests to look for similar estimates directly for the basepoint-freeness threshold $\eps_1(\l)$. Namely, one could ask if $\eps_1(\l)$ is less than or equal to 
\[
\Inf \{ r \in \Q^+ \ |\ (D_r^{\dim Z} \cdot Z) > (\dim Z)^{\dim Z}\ \textrm{for any abelian subvariety}\ \{0\} \neq Z \subseteq A \},
\] 
where $D_r := rL$ (see in particular \cite{ito}, Question 4.2). This is true for complex abelian surfaces, thanks to the Proposition \ref{nadel} and \cite{ito}.

The paper is organized as follows: 
in \S2 we recall the definition and some basic properties of cohomological rank functions, and show that, despite the fact that in \cite{jipa} the authors assume that the characteristic of the ground field is zero, the basic theory of cohomological rank functions works over an algebraically closed ground field of arbitrary characteristic as well. 
Finally, in this section we prove Proposition \ref{nadel}.

In \S3 we prove the basic properties of the threshold $\eps_1(\l)$ needed in the proof of the main results.

In \S4 we show a criterion, due to Kempf (\cite{ke}), reducing the property $(N_p)$ of syzygies to the surjectivity of certain multiplication maps of global sections, inductively defined. This is easily proved and well-known in characteristic zero (see e.g. \cite{el}, proof of Cor. 2.2, or \cite{pa1}, Lemma 4.1(a)). Kempf's approach is more complicated, but  has the advantage to work in arbitrary characteristic. Since Kempf's argument is somewhat obscure, we provide full details. We hope that this will be useful for extending to arbitrary characteristic some of known results concerning syzygies of projective varieties in characteristic zero. 

In \S5 we prove the Theorems \ref{main1} and \ref{syzmult}.

\vskip0.3truecm\noindent\textbf{Acknowledgments.} 
This work is part of my PhD thesis. I would like to thank my advisor, Giuseppe Pareschi, for his invaluable guidance.

\vskip0.3truecm\noindent\textbf{Notation.}
 Let $A$ be an abelian variety over an algebraically closed field, and let $\dim A = g$. For $b \in \Z$, 
\[
\mu_b : A \rightarrow A, \quad x \mapsto bx
\]
denotes the multiplication-by-$b$ isogeny of degree $b^{2g}$. A polarization $\l$ on $A$ is the class of an ample line bundle $L$ in $\NS(A)=\mathrm{Pic} A / \Pic0 A$. For a polarization $\l$ on $A$, the corresponding isogeny is denoted 
\[
\varphi_{\l} : A \rightarrow \widehat{A},
\]
where $\widehat{A} = \Pic0 A$ is the dual abelian variety. Recall that $\deg(\varphi_{\l}) = \chi(\l)^2 = (h^0(\l))^2$. We denote by $\PP$ the normalized Poincar\'e line bundle on $A \times \Ad$, and by $R\Phi_{\PP} : \mathrm{D}^b(A) \rightarrow \mathrm{D}^b(\Ad)$ the Fourier-Mukai-Poincar\'e equivalence (\cite{mukai}). Here $\mathrm{D}^b(A)$ denotes the bounded derived category of coherent sheaves on $A$. For $\alpha \in \Ad$, the corresponding line bundle on $A$ is denoted by $P_{\alpha} = \PP|_{A \times \{\alpha\}}$. Given a complex $\FF \in \mathrm{D}^b(A)$, we denote by $\FF^{\dual} = R\HHom(\FF, \OO_A)$ its derived dual, and by $h^i_{gen}(A, \FF)$ the dimension of the hypercohomology $H^i(A, \FF \tensor P_{\alpha})$, for $\alpha$ general in $\Ad$.

\section{Cohomological rank functions on abelian varieties}
Given $\FF \in \mathrm{D}^b(A)$, $i \in \Z$ and a polarization $\l$ on $A$, Jiang and Pareschi considered in \cite{jipa} cohomological rank functions
\[
h^i_{\FF,\, \l} : \Q \rightarrow \Q_{\geq 0}
\]
defined as follows
\[
h^i_{\FF,\, \l}(x) = h^i_{\FF}(x\l) := \frac{1}{b^{2g}}h^i_{gen}(A, \mu_b^*(\FF) \tensor L^{ab}),
\]
where $x = \frac{a}{b} \in \Q$ and $b > 0$. Since $\mu_b^*(\l) = b^2\l$, the pullback via $\mu_b$ of the rational class $\frac{a}{b}\l$ is $ab\l$. Moreover $\deg(\mu_b) = b^{2g}$, therefore, as explained in Remark 2.2 of \emph{op.cit.}, one may think of $h^i_{\FF, \l}(x)$ as the (generic) cohomological rank $h^i(A, \FF \la x\l \ra)$ of the $\Q$-\emph{twisted complex} $\FF \la x\l \ra$, which is defined similarly to \cite{laI}, \S6.2A. Namely, $\FF \la x \l \ra$ is the equivalence class of the pair $(\FF, x \l)$, where the equivalence is by definition
\[
(\FF \tensor L^m, x \l) \sim (\FF, (m+x)\l),
\]
for \emph{any} line bundle $L$ representing $\l$ and $m \in \Z$. Note that an ``untwisted'' object $\FF$ may be naturally seen as the $\Q$-twisted object $\FF \la 0\l \ra$. Moreover we have that $\FF \tensor P_{\alpha} \la x\l \ra = \FF \la x\l \ra$, for any $\alpha \in \Pic0(A)$.

In \cite{jipa} the authors introduced such notion assuming that the characteristic of the ground field $\K$ is zero. However the above definition makes sense in any characteristic. The main point consists in showing that it does not depend  on the representation $x = \frac{a}{b}$. To this purpose we need to verify that the quantity $h^i_{gen}(A, \FF)$ is multiplicative with respect to any isogeny $\mu_m$:
\begin{equation}\label{welldef}
h^i_{gen}(A, \mu_m^*\FF) = m^{2g}h^i_{gen}(A, \FF).
\end{equation}
This is checked in \emph{op.cit.} under the assumption that $\Char(\K)=0$. However the same thing can be checked removing such assumption as follows. By cohomology and base change, $h^i_{gen}(A, \mu_m^*\FF)$ is the generic rank of the Fourier-Mukai-Poincar\'{e} transform $R^i\Phi_{\PP}(\mu_m^*\FF)$. Moreover $R^i\Phi_{\PP}(\mu_m^*\FF) = \hat{\mu}_{m*}R^i\Phi_{\PP}(\FF)$ (\cite{mukai} (3.4)), where $\hat{\mu}_m : \Ad \rightarrow \Ad$ is the dual isogeny of $\mu_m$, i.e.\ it is the multiplication-by-$m$ isogeny of $\Ad$. Since the morphism $\hat{\mu}_m$ is in any case flat, the generic rank of $\hat{\mu}_{m*}R^i\Phi_{\PP}(\FF)$ is that of $R^i\Phi_{\PP}(\FF)$ multiplied by the degree of $\hat{\mu}_m$. Therefore we get (\ref{welldef}). Granting this, $h^i_{\FF}(x\l)$ is well-defined: if we take another representation of $x$, say $x = \frac{am}{bm}$, then
\begin{equation*}
\begin{split}
h^i_{\FF}(x\l) &= \frac{1}{(bm)^{2g}}h^i_{gen}(A, \mu_{bm}^*(\FF) \tensor L^{abm^2}) \\
               &= \frac{1}{(bm)^{2g}}h^i_{gen}(A, \mu_m^*(\mu_{b}^*(\FF) \tensor L^{ab})) \\
							 &= \frac{1}{b^{2g}}h^i_{gen}(A, \mu_{b}^*(\FF) \tensor L^{ab}).
\end{split}
\end{equation*} 
\begin{remark} Although we won't need this in this paper, we remark that from the above discussion it follows that the basic properties satisfied by the cohomological rank functions described in \S2 of \cite{jipa} -- especially the fundamental transformation formula with respect to the Fourier-Mukai-Poincar\'e transform  Prop. 2.3 of \emph{op.cit.} and its consequences -- work  in any characteristic.
\end{remark}

Using the cohomological rank functions it is possible to introduce several invariants attached to a polarized abelian variety $(A, \l)$. Let us recall that, given a line bundle $L$ that represents the class $\l$, the \emph{kernel bundle} $M_L$ associated to $L$ is by definition the kernel of the evaluation map $H^0(A, L) \tensor \OO_A \rightarrow L$. If $L$ is basepoint-free, then $M_L$ sits in the exact sequence
\begin{equation*}\label{defkb}
0 \rightarrow M_L \rightarrow H^0(A, L) \tensor \OO_A \rightarrow L \rightarrow 0.
\end{equation*}
\begin{definition}\label{invariants}
Let $(A, \l)$ be a polarized abelian variety. Then we consider
\[
\eps_1(\l) := \Inf \{ x \in \Q \ |\ h^1_{\II_p}(x\l) = 0 \},  
\]
where $\II_p$ is the ideal sheaf of a closed point $p \in A$ and, if $\l$ is basepoint-free
\[
\kappa_1(\l) := \Inf \{ x \in \Q \ |\ h^1_{M_L}(x\l) = 0 \}, 
\]
where $M_L$ is the kernel bundle associated to a line bundle $L$ representing $\l$.
\end{definition}
\begin{remark}
The above invariants are well-defined, i.e.\ $\eps_1(\l)$ does not depend on the point $p$, and $\kappa_1(\l)$ is independent from the representing line bundle $L$. We point out that -- although there no examples so far -- $\eps_1(\l)$ and $\kappa_1(\l)$ could be irrational. However, as will be clear later on, this does not create any trouble.
\end{remark} 
\noindent The relation between the above two constants was established by Jiang and Pareschi:
\begin{theorem}[\cite{jipa} Theorem D]\label{thmDjipa}
Let $\l$ be a basepoint-free polarization. Then
\[
\kappa_1(\l) = \frac{\eps_1(\l)}{1-\eps_1(\l)}.
\]
\end{theorem}
\begin{remark} From this result, in \emph{op.cit.} it is derived that $\kappa_1(\l) < 1$, i.e.\ $\l$ is projectively normal,  if and only if     $\eps_1(\l) < \frac{1}{2}$ (see in particular \cite{jipa}, Corollary 8.2 (b)). Our Theorem \ref{main1} is an extension of the ``if'' implication to higher syzygies. 
\end{remark}

\subsection{Proof of Proposition \ref{nadel}. }
Only in this subsection we make the assumption that the ground field $\K$ is $\C$.

Let $r \in \Q$ such that there exists an effective $\Q$-divisor $F$ on $A$ with 
\begin{subequations}
\begin{align}
r L &- F\ \textrm{ample}, \label{div1} \\ 
\II_0 &= \JJ(A, F). \label{div2}
\end{align}
\end{subequations}
To prove the Proposition we need to prove that 
\begin{equation}\label{vanish}h^1_{\II_0}(r\l)=0.
\end{equation}
Writing $r=\frac{a}{b}$ with $b > 0$, this means that 
\[h^1_{gen}(L^{ab}\otimes \mu_b^*\II_0)=0.\]
But, by (\ref{div2}), the left hand side is $h^1_{gen}(L^{ab}\otimes \mu_b^*\JJ(A, F))=h^1_{gen}(L^{ab}\otimes \JJ(A, \mu_b^*F))$, where we used that forming multiplier ideals commutes with pulling back under \'etale morphism (see \cite{laI}, Example 9.5.44). Since $\mu_b^*F \equiv_{\textrm{num}} b^2F$, it follows from (\ref{div1}) that $L^{ab} -\mu_b^*F$ is ample. Therefore (\ref{vanish}) follows from Nadel's vanishing.

\section{Generic vanishing of $\Q$-twisted sheaves on abelian varieties}

Following \S5 of \cite{jipa}, one can extend the usual notions of \emph{generic vanishing} to the $\Q$-twisted setting:
\begin{definition/theorem}[\cite{jipa} Theorem 5.1]\label{GVeq}
(1) A $\Q$-twisted \emph{sheaf} $\FF \la x\l \ra$, with $x = \frac{a}{b}$, is said to be $GV$ if
\[
\Codim_{\Ad}\ \Supp (R^i\Phi_{\PP}((\mu_b^*\FF) \tensor L^{ab})) \geq i, \quad \textrm{for all}\ i > 0. 
\]
Equivalently the transform\footnote{Recall that $R\Phi_{\PP^{\dual}}(\cdot) = (-1_{\Ad})^*R\Phi_{\PP}$.} $R\Phi_{\mathcal{P}^{\dual}}((\mu_b^*\mathcal{F}^{\dual}) \tensor L^{-ab})$ is a sheaf concentrated in cohomological degree $g$, i.e.
\[
R\Phi_{\PP^{\dual}}((\mu_b^*\FF^{\dual}) \tensor L^{-ab}) = R^g\Phi_{\PP^{\dual}}((\mu_b^*\FF^{\dual}) \tensor L^{-ab})[-g].
\]
(2) It is said to be $IT(0)$ if the transform
\[
R\Phi_{\mathcal{P}}((\mu_b^*\mathcal{F}) \tensor L^{ab}) = R^0\Phi_{\mathcal{P}}((\mu_b^*\mathcal{F}) \tensor L^{ab})
\] 
is concentrated in cohomological degree $0$.   
\end{definition/theorem}
\begin{remark}
(a) The above definitions do not depend on the representation $x = \frac{a}{b}$. For example for any $i$, $R^i\Phi_{\PP}(\mu_m^*\FF) = \hat{\mu}_{m*}R^i\Phi_{\PP}(\FF)$ (\cite{mukai} (3.4)) where $\hat{\mu}_m$ is the dual isogeny of $\mu_m$, therefore by cohomology and base change we see that $\Supp (R^i\Phi_{\PP}(\mu_m^*\FF))$ corresponds to the image of $\Supp (R^i\Phi_{\PP}(\FF))$ via the isogeny $\hat{\mu}_m$. \\
(b) They neither depend on the line bundle $L$ representing the class $\l$. Indeed, thanks to the exchange of translations and tensor product by elements of $\Pic0 A$ (\cite{mukai} (3.1)), if $L_0$ is another line bundle algebraically equivalent to $L$, then $R^i\Phi_{\PP}((\mu_b^*\FF) \tensor L_0^{ab})$ is a translate of $R^i\Phi_{\PP}((\mu_b^*\FF) \tensor L^{ab})$. 
\end{remark}
\noindent By cohomology and base change one has that 
\begin{equation}\label{incl1}
\begin{split}
\Supp (R^i\Phi_{\PP}((\mu_b^*\FF) \tensor L^{ab})) &\subseteq \{ \alpha \in \Ad \ |\ H^i(A, (\mu_b^*\FF) \tensor L^{ab} \tensor P_{\alpha}) \neq 0 \} \\   
           &=: V^i((\mu_b^*\FF) \tensor L^{ab}) 
\end{split}
\end{equation}
and, if $V^{i+1}((\mu_b^*\FF) \tensor L^{ab}) = \varnothing$, then equality holds. Moreover, we have that the $\Q$-twisted sheaf $\FF \la x\l \ra$ is $GV$ if and only if
\[
\Codim_{\Ad}\ V^i((\mu_b^*\FF) \tensor L^{ab}) \geq i, 
\]
for all $i > 0$ and for any representation $x = \frac{a}{b}$ (\cite{papogv} Lemma 3.6). By cohomology and base change again, $\FF \la x\l \ra$ is $IT(0)$ if and only if 
\[
V^i((\mu_b^*\FF) \tensor L^{ab}) = \varnothing
\]
for all $i > 0$ and for any representation $x = \frac{a}{b}$. In particular, we see that an $IT(0)$ $\Q$-twisted sheaf is $GV$.

These generic vanishing concepts are strongly related to the invariants introduced in Definition \ref{invariants}, as explained in \S8 of \cite{jipa}. Namely, we have 
\begin{lemma}[\cite{jipa}, p.\ 25]\label{lemmajipa}
Given two polarizations $\l$ and $\n$ -- with $\n$ basepoint-free -- and a rational number $x$, the fact that $\eps_1(\l) < x$ (resp.\ $\kappa_1(\n) < x$) is equivalent to the fact that the $\Q$-twisted sheaf $\II_p\la x\l \ra$ (resp.\ $M_N\la x\n \ra$) is $IT(0)$.
\end{lemma}
\noindent For reader's convenience we explicitly write down the case of $\eps_1(\l)$: assume that $\eps_1(\l) < x \in \Q$ and fix a sufficiently small $\eta > 0$ such that $x_0 := \eps_1(\l) + \eta \in \Q$ and $x_0 < x$. By (\ref{incl1}), $\II_p\la x_0 \l \ra$ is $GV$, therefore Hacon's criterion (see \cite{jipa}, Theorem 5.2 (a)) implies that $\II_p \la (x_0 + (x - x_0))\l \ra = \II_p \la x \l \ra$ is $IT(0)$. Conversely suppose that $\II_p \la x\l \ra$ is $IT(0)$, then $\II_p \la (x-y)\l \ra$ is still $IT(0)$, for a sufficiently small $y \in \Q^+$ (\cite{jipa} Theorem 5.2 (c)). Then $\eps_1(\l) < x - y < x$. For $\kappa_1(\n)$, the argument is similar.

The following is a $\Q$-twisted analog of a well known property of ``preservation of vanishing''(\cite{papoIII} Proposition 3.1).
\begin{proposition}\label{heartprop}
Assume that $\mathcal{F}$ and $\mathcal{G}$ are coherent sheaves, and that one of them is locally free. If $\mathcal{F}\langle x\l \rangle$ is $IT(0)$ and $\mathcal{G}\langle y\l \rangle$ is $GV$, then $\mathcal{F}\langle x\l \rangle \tensor \mathcal{G}\langle y\l \rangle := (\mathcal{F} \tensor \mathcal{G})\langle (x+y)\l \rangle$ is $IT(0)$.
\end{proposition}
\begin{proof}
Let $x = \frac{a}{b}$ and $y = \frac{c}{d}$, with $b, d > 0$. So $x+y = \frac{ad+bc}{bd}$. We want to prove that $\mu_{bd}^*(\mathcal{F} \tensor \mathcal{G}) \tensor L^{(ad+bc)bd}$ is an $IT(0)$ sheaf. By hypothesis $\mathcal{F}\langle x\l \rangle$ is $IT(0)$, hence 
\[
\mu_d^*((\mu_b^*\mathcal{F}) \tensor L^{ab}) = (\mu_{bd}^*\mathcal{F}) \tensor L^{abd^2}
\] 
is an $IT(0)$ sheaf, because $R\Phi_{\mathcal{P}}(\mu_d^*((\mu_b^*\mathcal{F}) \tensor L^{ab})) = \hat{\mu}_{d *}R\Phi_{\mathcal{P}}((\mu_b^*\mathcal{F}) \tensor L^{ab}) = \hat{\mu}_{d *}R^0\Phi_{\mathcal{P}}((\mu_b^*\mathcal{F}) \tensor L^{ab})$ (\cite{mukai} (3.4)) is concentrated in degree 0, where $\hat{\mu}_d : \widehat{A} \rightarrow \widehat{A}$ is the dual isogeny of $\mu_d$. Likewise, if $\mathcal{G}\langle y\l \rangle$ is $GV$, by using the equivalence in Definition \ref{GVeq} \emph{(1)}, we have that 
\[
\mu_b^*((\mu_d^*\mathcal{G}) \tensor L^{cd}) = (\mu_{bd}^*\mathcal{G}) \tensor L^{b^2cd}
\]
is a $GV$ sheaf. Since 
\[
\mu_{bd}^*(\mathcal{F} \tensor \mathcal{G}) \tensor L^{(ad+bc)bd} = ((\mu_{bd}^*\mathcal{F}) \tensor L^{abd^2}) \tensor ((\mu_{bd}^*\mathcal{G}) \tensor L^{b^2cd}),
\]
we conclude by applying the ``preservation of vanishing'' for (untwisted) coherent sheaves (\cite{papoIII} Proposition 3.1).   
\end{proof}    
  
For our purposes, the central result of this section is the following
\begin{proposition}\label{mainprop}
Let $p$ be a non-negative integer. If 
\[
\eps_1(\l) < \frac{1}{p+2},
\]
then $M_L^{\otimes(p+1)} \tensor L^h$ is $IT(0)$ for all $h \geq 1$.
\end{proposition}
\begin{proof}
Let $L$ be a line bundle representing $\l$, and let $M_L$ be the kernel of the evaluation morphism $H^0(A, L) \tensor \OO_A \rightarrow L$. The assumption on $\eps_1(\l)$ implies, in particular, that $\l$ is basepoint-free and, using Theorem \ref{thmDjipa}, we get
\begin{align*}
\kappa_1(\l) &= \frac{\eps_1(\l)}{1 - \eps_1(\l)} \\
             &= -1 + \frac{1}{1-\eps_1(\l)} \\
						&< -1 + \frac{p+2}{p+1} \\
						&= \frac{1}{p+1}.
\end{align*}
By Lemma \ref{lemmajipa}, this is equivalent to say that $M_L\la \frac{1}{p+1} \l \ra$ is an $IT(0)$ $\Q$-twisted sheaf. Fix now an integer $h \geq 1$ and write $M_L^{\otimes(p+1)} \tensor L^h = M_L^{\otimes(p+1)} \tensor L \tensor L^{h-1}$ as the $\Q$-twisted sheaf
\[
M_L^{\otimes(p+1)} \la (\frac{p+1}{p+1} + h-1) \l \ra = \bigl( M_L\la \frac{1}{p+1} \l \ra \bigr)^{\otimes(p+1)} \tensor \OO_A \la (h-1)\l \ra.  
\]
Since $L^{h-1}$ is ample -- hence $IT(0)$ -- if $h > 1$, or it is trivial if $h =1$, and $M_L\la \frac{1}{p+1} \l \ra$ is $IT(0)$, we have that $M_L^{\otimes(p+1)} \tensor L^h$ is $IT(0)$ thanks to the ``preservation of vanishing'' (Proposition \ref{heartprop}).
\end{proof}

\section{Syzygies and the property $(N_p)$}
We recall the definition and geometric meaning of the property $(N_p)$ in more detail. Let $X$ be a projective variety, defined over an algebraically closed field $\K$. If $L$ gives an embedding
\[
\phi_{|L|} : X \hookrightarrow \P = \P(H^0(X, L)^{\dual}),
\]   
then $L$ is said to \emph{satisfy the property} $(N_p)$ if the first $p$ steps of the minimal graded free resolution $E_{\bullet}(L)$ of the algebra $R_L := \bigoplus_m H^0(X, L^m)$ over the polynomial ring $S_L := \Sym\ H^0(X, L)$ are linear, i.e.\ of the form
\[
\xymatrix{
S_L(-p-1)^{\oplus i_p} \ar[r] \ar@{=}[d]  & S_L(-p)^{\oplus i_{p-1}} \ar[r] \ar@{=}[d] & \ldots \ar[r] & S_L(-2)^{\oplus i_1} \ar[r] \ar@{=}[d] & S_L \ar[r] \ar@{=}[d] & R_L \ar[r] & 0 \\
E_p(L)                                                & E_{p-1}(L)                                              &                & E_1(L)                                 & E_0(L)                         &               &
}
\]
Thus $(N_0)$ means that $L$ is projectively normal (and in this case a resolution of the homogeneous ideal $I_{X/\P}$ of $X$ in $\P$ is given by $\ldots \rightarrow E_1(L) \rightarrow I_{X/\P} \rightarrow 0$); $(N_1)$ means that $I_{X/\P}$ is generated by quadrics; $(N_2)$ means that the relations among these quadrics are generated by linear ones and so on.

Writing $\K = S_L/S_{L +}$ as the quotient of the polynomial ring $S_L$ by the irrelevant maximal ideal $S_{L +} := \bigoplus_{m\geq1} \Sym^m H^0(X, L)$, it is well known that $\dim_{\K} (\Tor_i^{S_L}(R_L, \K)_{j})$ computes the cardinality of any minimal set of homogeneous generators of $E_i(L)$ of degree $j$, therefore
\[
E_i(L) = \bigoplus_j \Tor_i^{S_L}(R_L, \K)_{j} \tensor_{\K} S_L(-j) 
\]
and $L$ satisfies the property $(N_p)$ if and only if 
\begin{equation}\label{tor}
\Tor_p^{S_L}(R_L, \K)_{j} = 0 \quad  \textrm{for all}\ j \geq p+2.\footnote{$\Tor_0^{S_L}(R_L, \K)_1$ is always trivial, because we are dealing with the complete linear series $|L|$ and the corresponding embedding is linearly normal. 
Moreover, the vanishing $\Tor_p^{S_L}(R_L, \K)_{j} = 0$ for all $j \geq p+2$, forces $\Tor_i^{S_L}(R_L, \K)_{j} = 0$ for all $0 \leq i \leq p$ and $j \geq i+2$ (see the proof of Proposition 1.3.3 in \cite{lasampl} for details).}  
\end{equation}

A well established condition ensuring the property $(N_p)$ for $L$ in \emph{characteristic zero} is the vanishing
\begin{equation}\label{charzero1}
H^1(X, M_L^{\otimes(p+1)} \tensor L^h) = 0 \quad \textrm{for all}\ h \geq 1.
\end{equation}
Indeed, tensoring the Koszul resolution of $\K$ by $R_L$ and taking graded pieces, we see that the property $(N_p)$ for $L$ is equivalent to the exactness in the middle of the Koszul complex 
\begin{equation*}\label{kc}
\Lambda^{p+1}H^0(X, L) \tensor H^0(X, L^h) \rightarrow \Lambda^{p}H^0(X, L) \tensor H^0(X, L^{h+1}) \rightarrow \Lambda^{p-1}H^0(X, L) \tensor H^0(X, L^{h+2})
\end{equation*}
for all $h \geq 1$ (see \cite[pp.\ 510--511]{lasampl} for details). This can be expressed in terms of the kernel bundle of $L$. Namely, taking wedge products of the exact sequence 
\[
0 \rightarrow M_L \rightarrow H^0(X, L) \tensor \OO_X \rightarrow L \rightarrow 0,
\] 
we get
\[
0 \rightarrow \Lambda^{p+1}M_L \rightarrow \Lambda^{p+1}H^0(X, L) \tensor \OO_X \rightarrow \Lambda^{p}M_L \tensor L \rightarrow 0. 
\]
Tensoring it by $L^{h}$ and taking global section, we see that the exactness of the Koszul complex above is equivalent to the surjectivity of the map
\[
\Lambda^{p+1}H^0(X, L) \tensor H^0(X, L^h) \rightarrow H^0(X, \Lambda^{p}M_L \tensor L^{h+1}),
\]
that in turn follows from the vanishing
\begin{equation}\label{charzero2}
H^1(X, \Lambda^{p+1}M_L \tensor L^h) = 0 \quad \textrm{for all}\ h \geq 1.
\end{equation}
Now, if $\Char(\K) = 0$, $\Lambda^{p+1}M_L$ is a \emph{direct summand} of $M_L^{\otimes(p+1)}$ and in particular (\ref{charzero1}) implies (\ref{charzero2}); otherwise said $L$ satisfies the property $(N_p)$. If $\Char(\K) > 0$, the exterior power $\Lambda^{p+1}M_L$ may no longer be a direct summand of the tensor power $M_L^{\otimes(p+1)}$, hence the above discussion does not apply. Nevertheless in this section, following an approach essentially due to G. Kempf, we prove that (\ref{charzero1}) implies the property $(N_p)$ for $L$, even in \emph{positive characteristic}:
\begin{proposition}
Let $X$ be a projective variety defined over an algebraically closed field $\mathbb{K}$. Let $L$ be an ample and globally generated line bundle on $X$, and let $p$ be a non-negative integer. If
\[
H^1(X, M_L^{\otimes(p+1)} \otimes L^h) = 0 \quad \textrm{for all}\ h \geq 1,
\]
then the property $(N_p)$ holds for $L$. 
\end{proposition}

Let us start by recalling two definitions and an algebraic lemma of Kempf (\cite{ke}, see also \cite[\S2]{ru}):
\begin{definition}\label{def1}
For any $L_i$'s (not necessarily ample) line bundles on $X$, let $K(L_1) = H^0(X, L_1)$ and, for $n > 1$, define $K(L_1, \ldots , L_n)$ inductively by the exact sequence:
\[
0 \rightarrow K(L_1, \ldots , L_n) \rightarrow K(L_1, L_3, \ldots , L_n) \tensor K(L_2) \rightarrow K(L_1 \tensor L_2, L_3, \ldots , L_n).
\]
\end{definition}
In particular, $K(L_1, L_2)$ is the kernel of the multiplication map of global sections $H^0(X, L_1) \tensor H^0(X, L_2) \rightarrow H^0(X, L_1 \tensor L_2)$.  
\begin{definition}
Let $S$ be a polynomial ring over $\K$ and let $R$ be a finitely generated graded $S$-module. \\
(1) Define $T^0(R) := R$, $T^1(R) := \Ker[R(-1) \tensor_{\K} S_1 \rightarrow R]$ and inductively
\begin{align*}
T^i(R) := T^{i-1}(T^1(R)).
\end{align*}
(2) Define 
\[
d^i(R) := \min \{d \in \Z \ |\ T^i(R)\ \textrm{is generated over $S$ by elements of degree} \leq d \}.  
\]
\end{definition}  
\begin{lemma}[Kempf \cite{ke}, Lemma 16]\label{lemmak}
Let $S = \K[x_0, \ldots , x_r]$ be a polynomial ring, graded in the standard way, over $\K = S/(x_0, \ldots , x_r)$. Let $R$ be a finitely generated graded $S$-module. If $j > p -i + d^i(R)$ for all $0 \leq i \leq p$, then
\[
\Tor_p^{S}(R, \K)_{j} = 0.
\]
\end{lemma}
\noindent Due to some obscurities in Kempf's argument and for the sake of self-containedness, we prefer to give a proof of the above Lemma, which closely follows that of Kempf.
\begin{proof}[Proof of Lemma \ref{lemmak}]
Consider the exact sequence 
\[
0 \rightarrow T^1(R) \rightarrow R(-1) \tensor_{\K} S_1 \stackrel{\alpha}{\rightarrow} R.
\]
The image $R'$ of $\alpha$ is a graded submodule of $R$. The quotient module $Q = R/R'$ is of finite length, hence its Castelnuovo-Mumford regularity $\reg(Q) = \max \{ d \ |\ Q_d \neq 0 \}$ (see \cite{eis}, Corollary 4.4). Moreover $Q$ is zero in degrees $> d^0(R)$, therefore
\begin{equation}\label{keq1}
\Tor_p^S(Q, \K)\ \textrm{is zero in degrees} > p + d^0(R).
\end{equation}
Indeed, if $\Tor_p^S(Q, \K)_j \neq 0$ for a $j > p + d^0(R)$, then $\reg(Q) \leq d^0(R) < j - p$. But, by definition, $\reg(Q) = \Sup \{k-i \ |\ \dim_{\K}( \Tor_i^S(Q, \K)_k) \neq 0 \}$ and so we get a contradiction. Now (\ref{keq1}) implies that the map
\begin{equation*}\label{keq2}
\Tor_p^S(R', \K) \rightarrow \Tor_p^S(R, \K)
\end{equation*}
is surjective in degrees $> p + d^0(R)$. Therefore, in order to prove the statement, it is enough to prove that $\Tor_p^S(R', \K)_j = 0$, if $j > p + d^0(R)$. From the long exact sequence associated to
\[
0 \rightarrow T^1(R) \rightarrow R(-1) \tensor_{\K} S_1 \stackrel{\alpha}{\rightarrow} R' \rightarrow 0,
\]
we get
\[
\Tor_p^S(R(-1) \tensor_{\K} S_1, \K) \stackrel{\alpha_*}{\rightarrow} \Tor_p^S(R', \K) \stackrel{\delta}{\rightarrow} \Tor_{p-1}^S(T^1(R), \K).
\]
Note that $\alpha_*$ is the multiplication by $S_1$ in the first variable. Since $\alpha_*$ is also the multiplication by $S_1$ in the second variable, it is the zero map. Therefore $\delta$ gives an inclusion
\[
\Tor_p^S(R', \K) \subseteq \Tor_{p-1}^S(T^1(R), \K)
\] 
and we may repeat this procedure $p$ times, obtaining 
\[
\Tor_{-1}^S(T^{p+1}(R), \K) = 0.
\]
\end{proof}

If now $L$ is an ample line bundle on $X$, $S = S_L$ and $R = R_L$, the link between the previous definitions is given by  
\begin{equation}\label{link}
T^i(R_L) = \bigoplus_{m \geq i} K(L^{m-i}, \underbrace{L, \ldots , L}_{i}).
\end{equation}
\begin{proof} If $i = 0$, then $T^0(R_L) = R_L$ and $K(L^m) = H^0(X, L^m)$. So (\ref{link}) is true. By definition
\[
T^i(R_L) = T^{i-1}(T^1(R_L)) = T^{i-1}(\Ker[R_L(-1) \tensor_{\K} H^0(X, L) \rightarrow R_L]),
\]  
and 
\[
0 \rightarrow \bigoplus_{m \geq i} K(L^{m-i}, \underbrace{L, \ldots , L}_{i}) \rightarrow \bigoplus_{m \geq i} K(L^{m-i}, \underbrace{L, \ldots , L}_{i-1}) \tensor H^0(X, L) \rightarrow \bigoplus_{m \geq i} K(L^{m-i+1}, \underbrace{L, \ldots , L}_{i-1}). 
\]
Therefore (\ref{link}) holds, by induction on $i$.
\end{proof}
\noindent The next Lemma allows to reduce the property $(N_p)$ for $L$ to the vanishing (\ref{charzero1}), in a way that avoids the exterior power of $M_L$.
\begin{lemma}\label{lemmalink}
(1) For all $n \geq 0$ and $h \geq 1$, one has $H^0(X, M_L^{\otimes n} \tensor L^h) = K(L^h, \underbrace{L, \ldots , L}_{n})$, if $L$ is basepoint-free. \\
(2) Let $i \geq 0$ and $h \geq 1$. If $L$ is basepoint-free and $H^1(X, M_L^{\otimes(i+1)} \tensor L^h) = 0$, then the multiplication map
\[
K(L^h, \underbrace{L, \ldots , L}_{i}) \tensor H^0(X, L) \rightarrow K(L^{h+1}, \underbrace{L, \ldots , L}_{i})
\]
is surjective. \\
(3)\emph{(Rubei \cite{ru}, p.\ 2578).} 
If the multiplication maps
\[
K(L^h, \underbrace{L, \ldots , L}_{i}) \tensor H^0(X, L) \rightarrow K(L^{h+1}, \underbrace{L, \ldots , L}_{i})
\]
are surjective for all $h \geq 1$, then $d^i(R_L) = i+1$.
\end{lemma}
\begin{proof}
\emph{(1)}\,: If $n = 0$, then by definition $H^0(X, L^h) = K(L^h)$ for all $h \geq 1$. Suppose $n \geq 1$. The kernel bundle $M_L$ sits in the short exact sequence
\begin{equation}\label{mladd}
0 \rightarrow M_L \rightarrow H^0(X, L) \tensor \OO_X \rightarrow L \rightarrow 0.
\end{equation}
Tensoring it by $M_L^{\otimes(n-1)}\tensor L^h$, one obtains
\begin{equation}\label{ml1}
0 \rightarrow M_L^{\otimes n} \tensor L^h \rightarrow H^0(X, L) \tensor M_L^{\otimes(n-1)}\tensor L^h \rightarrow M_L^{\otimes(n-1)}\tensor L^{h+1} \rightarrow 0.
\end{equation}
Taking global sections of (\ref{ml1}) and using the inductive hypothesis, we obtain
\[
0 \rightarrow H^0(X, M_L^{\otimes n} \tensor L^h) \rightarrow H^0(X, L) \tensor K(L^h, \underbrace{L, \ldots , L}_{n-1}) \rightarrow K(L^{h+1}, \underbrace{L, \ldots , L}_{n-1}).  
\]
Therefore, by definition, $H^0(X, M_L^{\otimes n} \tensor L^h) = K(L^h, \underbrace{L, \ldots , L}_{n})$. \\
\emph{(2)}\,: 
Tensoring (\ref{mladd}) by $M_L^{\otimes i} \tensor L^h$, we have
\begin{equation}\label{ml2}
0 \rightarrow M_L^{\otimes(i+1)} \tensor L^h \rightarrow H^0(X, L) \tensor M_L^{\otimes i}\tensor L^h \rightarrow M_L^{\otimes i}\tensor L^{h+1} \rightarrow 0.
\end{equation}
From the long exact sequence in cohomology associated to (\ref{ml2}), and thanks to the point \emph{(1)}, one has
\[
H^0(X, L) \tensor K(L^h, \underbrace{L, \ldots , L}_{i}) \stackrel{\alpha}{\rightarrow} K(L^{h+1}, \underbrace{L, \ldots , L}_{i}) \rightarrow H^1(X, M_L^{\otimes(i+1)} \tensor L^h) = 0.
\]
Therefore the multiplication map $\alpha$ is surjective. \\
\emph{(3)}\,: By (\ref{link}) and the hypothesis we have that $T^i(R_L)$ is generated over $S_L$ by 
\[
K(\underbrace{L, \ldots , L}_{i+1}).
\]
This means that it is generated by the piece of degree $m$ with $m-i = 1$, i.e.\ $m=i+1$. Therefore $d^i(R_L) = i+1$.
\end{proof}

\section{Proof of the Theorems \ref{main1} and \ref{syzmult}}

\begin{proof}[Proof of Theorem \ref{main1}]
Let $L$ be a representative of the class $\l$. For all $0 \leq i \leq p$, we have 
\[
\eps_1(\l) < \frac{1}{p+2} \leq \frac{1}{i+2}.
\]
Therefore $L$ is basepoint-free and, thanks to the Proposition \ref{mainprop}, we know that $M_L^{\otimes(i+1)} \tensor L^h$ is $IT(0)$, for all $h \geq 1$. This implies, in particular, that $H^1(A, M_L^{\otimes(i+1)} \tensor L^h) = 0$ for all $h \geq 1$. Hence, by Lemma \ref{lemmalink} \emph{(2)} and \emph{(3)}, we obtain 
\[
d^i(R_L) = i+1.
\]
Now, if $j > p -i + d^i(R_L) = p + 1$, Kempf's Lemma \ref{lemmak} implies that
\[
\Tor_p^{S_L}(R_L, \K)_{j} = 0.
\]
As explained in (\ref{tor}), this is equivalent to the property $(N_p)$ for $L$.
\end{proof}

\begin{proof}[Proof of Theorem \ref{syzmult}]
Note that we have already proved the $t=0$ case -- even without the basepoint-freeness assumption -- and the $t=1$ case (Corollary \ref{pathm}). Hence we may assume $t > 1$. By Theorem \ref{main1}, it suffices to show that $\eps_1(m\l) < \frac{1}{p+2}$. We have
\[
\eps_1(m\l) = \frac{\eps_1(\l)}{m} \leq \frac{\eps_1(\l)}{p+3-t} \leq \frac{1}{t(p+3-t)},
\] 
where the last inequality follows by definition. Let us impose now the inequality
\[
\frac{1}{t(p+3-t)} < \frac{1}{p+2},
\]
or equivalently
\[
t^2 - (p+3)t + p+2 < 0.
\]
This is satisfied if and only if $1 < t <  p+2$ and, by hypothesis, we have $1 < t \leq  p+1$.
\end{proof}

\providecommand{\bysame}{\leavevmode\hbox
to3em{\hrulefill}\thinspace}


\begin{thebibliography}{EMS}


\bibitem[EL]{el} L. Ein and R. Lazarsfeld, {Syzygies and Koszul cohomology
of smooth projective varieties of arbitrary dimension}, Invent. Math. \textbf{111} (1993), no. 1, 51--67.

\bibitem[Ei]{eis} D. Eisenbud, {\em The geometry of syzygies}, Springer-Verlag, New York, 2005.



\bibitem[Gr]{green} M. Green, {Koszul cohomology and the geometry of projective varieties}, J. Differ. Geom. \textbf{19} (1984), no. 1, 125--171.



\bibitem[GL]{grla} M. Green and R. Lazarsfeld, {On the projective normality of complete linear series on an algebraic curve}, Invent. Math. \textbf{83} (1986), no. 1, 73--90.

\bibitem[HT]{ht} J.M. Hwang, W.K. To, {Buser-Sarnak invariant and projective normality of abelian varieties}, in {\em Complex and differential geometry}, Springer Proc. Math., vol. 8, Springer, Heidelberg, 2011, 157--170. 



\bibitem[It]{ito} A. Ito, {A remark on higher syzygies on abelian surfaces}, Comm. Algebra \textbf{46} (2018), no. 12, 5342--5347.




\bibitem[JP]{jipa} Z. Jiang and G. Pareschi, {Cohomological rank functions on abelian varieties}, preprint arXiv:1707.05888 (2017), to appear on Ann. Sci. \'{E}cole Norm. Sup. 


\bibitem[Ke1]{ke} G. Kempf, {Projective coordinate rings of abelian varieties}, in {\em Algebraic analysis, geometry and number theory}, Johns Hopkins Univ. Press, Baltimore, MD, 1989, 225--235.

\bibitem[Ke2]{ke2} G. Kempf, {\em Complex Abelian Varieties and Theta Functions}, Springer-Verlag, Berlin, 1991.

\bibitem[Ko]{ko} S. Koizumi, {Theta relations and projective normality of Abelian varieties}, Amer. J. Math. \textbf{98} (1976), no. 4, 865--889.

\bibitem[KL]{kulo} A. K\"uronya and V. Lozovanu, {A Reider-type theorem for higher syzygies on abelian surfaces}, 
Algebr. Geom. \textbf{6} (2019), no. 5, 548--570.


\bibitem[La1]{lasampl} R. Lazarsfeld,  {A sampling of vector bundle techniques in the study of linear series}, in {\em Lectures on Riemann surfaces}, World Sci. Publ., Teaneck, NJ, 1989, 500--559.



\bibitem[La2]{laI} R. Lazarsfeld, {\em Positivity in algebraic geometry II}, Springer-Verlag, Berlin, 2004.




\bibitem[LPP]{lapapo} R. Lazarsfeld, G. Pareschi and M. Popa,  {Local positivity, multiplier ideals, and syzygies of abelian varieties}, Algebra Number Theory \textbf{5} (2011), no. 2, 185--196.


\bibitem[Lo]{lo} V. Lozovanu, {Singular divisors and syzygies of polarized abelian threefolds}, preprint  arXiv:1803.08780 (2018). 




\bibitem[M]{mukai} S. Mukai, {Duality between D($X$) and D($\hat X$) with its application to Picard sheaves}, Nagoya Math. J. \textbf{81} (1981), 153--175.

\bibitem[Mu]{mum} D. Mumford, {Varieties defined by quadratic equations}, in {\em Questions on Algebraic Varieties (C.I.M.E., III Ciclo, Varenna, 1969)}, Edizioni Cremonese, Rome, 1970, 29--100. 




\bibitem[Pa]{pa1} G. Pareschi,  {Syzygies of abelian varieties}, J. Amer. Math. Soc. \textbf{13} (2000), no. 3, 651--664.





\bibitem[PP1]{papoII} G. Pareschi and M. Popa, {Regularity of abelian varieties II: basic results on linear series and defining equations}, J. Alg. Geom. \textbf{13} (2004), no. 1, 167--193.

\bibitem[PP2]{papoIII} G. Pareschi and M. Popa, {Regularity on abelian varieties III: relationship with generic vanishing and applications}, in {\em Grassmannians, moduli spaces and vector bundles}, Clay Math. Proc., \textbf{14}, Amer. Math. Soc., Providence, RI, 2011, 141--167.

\bibitem[PP3]{papogv} G. Pareschi and M. Popa,  {$GV$-sheaves, Fourier-Mukai transform, and generic vanishing}, Amer. J. Math. \textbf{133} (2011), no. 1, 235--271.


\bibitem[Ru]{ru} E. Rubei,  {On syzygies of abelian varieties}, Trans. Amer. Math. Soc. \textbf{352} (2000), no. 6, 2569--2579.

\bibitem[Sa]{sa} R. Sasaki,  {Theta relations and their applications in abstract geometry}, Sci. Rep. Tokyo Kyoiku Daigaku Sect. A \textbf{13} (1977), no. 366--382, 290--312.
 
\bibitem[Se1]{se1} T. Sekiguchi,  {On projective normality of Abelian varieties}, J. Math. Soc. Japan \textbf{28} (1976), no. 2, 307--322.

\bibitem[Se2]{se2} T. Sekiguchi,  {On projective normality of Abelian varieties II}, J. Math. Soc. Japan \textbf{29} (1977), no. 4, 709--727.






\end{thebibliography}
\end{document}